\documentclass[10pt]{amsart}
\usepackage{amssymb}
\usepackage{amsbsy}
\usepackage{amscd}
\usepackage[mathscr]{eucal}
\usepackage{verbatim}
\usepackage{version}
\usepackage{color}
\usepackage[matrix,arrow]{xy}

\def\cal{\mathcal}
\def\Bbb{\mathbb}

\newenvironment{NB}{
\color{red}{\bf NB}. \footnotesize 
}{}
\excludeversion{NB}
\newenvironment{NB2}{
\color{blue}{\bf NB}. \footnotesize
}{}

\newcommand{\bb} {\mathbb}
\newcommand{\bl} {\mathbf}

\newcommand{\simto}{\xrightarrow{\ \sim\ }}

\newcommand{\Eq}  {\operatorname{Eq}}
\newcommand{\FM}  {\operatorname{FM}}

\newcommand{\LieO}{\operatorname{O}}

\newcommand{\SL}  {\operatorname{SL}}

\newcommand{\Pic}{\operatorname{Pic}}

\newcommand{\ch}{\operatorname{ch}}

\newcommand{\Coh}{\operatorname{Coh}}

\newcommand{\Ext}{\operatorname{Ext}}
\newcommand{\Hom}{\operatorname{Hom}}

\newcommand{\Aut}{\operatorname{Aut}}
\newcommand{\rk}{\operatorname{rk}}

\newcommand{\NS}{\operatorname{NS}}

\newcommand{\Alb}{\operatorname{Alb}}

\newcommand{\Amp}{\operatorname{Amp}}

\newcommand{\tr}{\operatorname{tr}}
\newcommand{\alg}{\operatorname{alg}}

\newcommand{\End}{\operatorname{End}}

\newcommand{\Sym}{\operatorname{Sym}}

\font\b=cmr10 scaled \magstep5
\def\bigzerou{\smash{\lower1.7ex\hbox{\b 0}}}

\numberwithin{equation}{section}

\theoremstyle{plain}
 \newtheorem{thm}{Theorem}[section]
 \newtheorem{lem}[thm]{Lemma}
 \newtheorem{prop}[thm]{Proposition}
 \newtheorem{cor}[thm]{Corollary}
 
\theoremstyle{definition}
 \newtheorem{defn}[thm]{Definition}
 
\theoremstyle{remark}
 \newtheorem{rem}[thm]{Remark}

\begin{document}

\title{Categorical entropy for Fourier-Mukai transforms
on generic abelian surfaces.}
\author{K\={o}ta Yoshioka}
\address{Department of Mathematics, Faculty of Science,
Kobe University, Kobe, 657, Japan}
\email{yoshioka@math.kobe-u.ac.jp}

\thanks{
The author is supported by the Grant-in-aid for 
Scientific Research (No.\ 26287007,\ 24224001), JSPS}
\keywords{Categorical entropy, abelian surfaces}
\subjclass[2010]{Primary 14D20}

\begin{abstract}
In this note, we shall compute the categorical entropy of 
an autoequivalence on a generic abelian surface.
\end{abstract}

\maketitle

\section{Introduction}
In \cite{DHKK}, Dimitrov, Haiden, Katzarkov, and Kontsevich introduced
a categorical entropy $h_t(\Phi)$ $(t \in {\Bbb R})$ 
for an endofunctor $\Phi$ of a triangulated category
with a split generator. For endofunctors of the derived category ${\bf D}(X)$
of coherent sheaves
on a smooth projective variety $X$,
Kikuta and Takahashi \cite{Kikuta}, \cite{KT} studied the entropy.
In particular they proved that $h_0$ coincides with the topological entropy
for 
an endofunctor ${\bf L}f^*$ induced by a surjective endomorphism 
$f$ of smooth projective variety \cite[Thm. 5.4]{KT} 
and an autoequivalence of ${\bf D}(X)$ 
if $\dim X=1$ \cite{Kikuta} or $\pm K_X$ is ample \cite[Thm. 5.6]{KT}.
They also conjectured that a Gromov-Yomdin type result holds, that is,
$h_0(\Phi)=\log \rho(\Phi)$ (\cite[Conjecture 5.3]{KT}) where
$\rho(\Phi)$ is the spectral radius of the action of $\Phi$ on the algebraic 
cohomology group $H^*(X,{\Bbb Q})_{\alg}$. 
It seems that there are only a few example of  computation of 
categorical entropy, and it may be interesting to add more examples. 
In this note, we shall give an almost trivial
 example of the computation.
Thus we shall compute the entropy
for special autoequivalences on abelian surfaces.
For an abelian variety, Orlov \cite{Or1} 
proved that the kernel of an equivalence
is a sheaf up to shift. So we can expect that the entropy
which measures the complexity of an equivalence
is simple. For a special equivalence on
an abelian surface, we shall check that 
our expectation is true.

To be more precise,
let $X$ be an abelian surface and $H$ an ample divisor.
We set $L:={\Bbb Z} \oplus {\Bbb Z}H \oplus {\Bbb Z}\varrho_X$,
where $\varrho_X$ is the fundamental class of $X$.
Let $\Phi:{\bf D}(X) \to {\bf D}(X)$ be a Fourier-Mukai functor
which preserves $L$.
In Theorem \ref{thm:h_t},
we shall compute $h_t(\Phi)$.
Combining a recent paper of Ikeda \cite{Ikeda},
we also check that
the conjecture of Kikuta and Takahashi holds for equivalences
on abelian surfaces (Proposition \ref{prop:KT}).

We also remark that there is a symplectic manifold of 
generalized Kummer type which has an automorphism of infinite order.
This is an analogue of a recent result of Ouchi \cite{Ouchi}.
By the absence of spherical objects,
our example is almost trivial.

\section{Fourier-Mukai transforms on an abelian surface}
\subsection{Notation.}
We denote the category of coherent sheaves on $X$ by $\Coh(X)$ and
the bounded derived category of $\Coh(X)$ by ${\bf D}(X)$.
A Mukai lattice of $X$ consists of 
$H^{2*}(X,{\Bbb Z}):=\bigoplus_{i=0}^2 H^{2i}(X,{\Bbb Z})$ 
and an integral bilinear form
$\langle\;\;,\;\; \rangle$ on 
$H^{2*}(X,{\Bbb Z})$:
$$
\langle x_0+x_1+x_2 \varrho_X,y_0+y_1+y_2 \varrho_X \rangle:=
(x_1,y_1)-x_0 y_2-x_2 y_0  \in {\Bbb Z},
$$
where $x_1,y_1 \in H^{2}(X,{\Bbb Z})$, 
$x_0,x_2,y_0,y_2 \in {\Bbb Z}$ and $\varrho_X \in H^4(X,{\Bbb Z})$ 
is the fundamental class of $X$.
We also introduce the algebraic Mukai lattice as
the pair of $H^*(X,{\Bbb Z})_{\alg}:=
{\Bbb Z} \oplus \NS(X) \oplus {\Bbb Z}$
and $\langle\;\;,\;\; \rangle$ on 
$H^*(X,{\Bbb Z})_{\alg}$.
For $x=x_0+x_1+x_2 \varrho_X$ with $x_0,x_2 \in {\Bbb Z}$
and $x_1 \in H^2(X,{\Bbb Z})$,
we also write $x=(x_0,x_1,x_2)$. 
For $E \in {\bf D}(X)$,
$v(E):=\ch(E)$ denotes the Mukai vector of $E$.

For ${\bf E} \in {\bf D}(X \times Y)$,
we set 
$$
\Phi_{X \to Y}^{{\bf E}}(x):={\bf R} p_{Y*}({\bf E} \otimes p_X^*(x)),\;
x \in {\bf D}(X),
$$
where $p_X,p_Y$ are projections from $X \times Y$ to
$X$ and $Y$ respectively. 
Let $\Eq(\bl{D}(X),\bl{D}(Y))$
be the set of equivalences between $\bl{D}(X)$ and $\bl{D}(Y)$.
We set 
\begin{align*}
&\Eq_0(\bl{D}(Y),\bl{D}(Z))\\
:=&
 \left\{ \left. \Phi_{Y \to Z}^{\bl{E}[2k]} 
   \in \Eq(\bl{D}(Y),\bl{D}(Z))
   \right| \bl{E} \in \Coh(Y \times Z),\, k \in {\bb{Z}} 
 \right\},
\\
& \cal{E}(Z)
:=
\bigcup_{Y}\Eq_0(\bl{D}(Y),\bl{D}(Z)),\\
& \cal{E}
:=
 \bigcup_{Z}\cal{E}(Z)
=\bigcup_{Y,Z }\Eq_0(\bl{D}(Y),\bl{D}(Z)).
\end{align*}
Note that $\cal{E}$ is a groupoid with respect to the composition 
of the equivalences.

For an object $E \in {\bf D}(X)$ with $\rk E \ne 0$,
 we set $\mu(E):=c_1(E)/\rk E$.

\subsection{Semi-homogeneous sheaves}

We collect some properties of semi-homogeneous sheaves on an abelian surface
\cite{Mu0}.

\begin{prop}\label{prop:stable}
For a coherent sheaf $E$ on $X$, the following conditions are equivalent.  
\begin{enumerate}
\item
$E$ is a semi-homogeneous sheaf.
\item
$E$ is a semi-stable sheaf with $\langle v(E)^2 \rangle=0$ with respect to 
an ample divisor $H$.
\end{enumerate}
\end{prop}

\begin{NB}
Let $E$ be a 1-dimensional semi-homogeneous sheaf.
Then $\Phi_{X \to \Pic^0(X)}^{{\bf P}}(E \otimes H^{\otimes n})$
is a semi-homogeneous vector bundle on $\Pic^0(X)$.
Hence $E$ is a semi-stable sheaf with 
$\langle v(E)^2 \rangle=0$.
(If $n$ is sufficiently large, then stability is preserved.
So $E$ must be semi-stable.)

Then $(c_1(E)^2)=0$. Indeed if $(c_1(E)^2)>0$, then
$x \mapsto T_x^*(\det E) \otimes (\det E)^{-1}$
gives a surjective map $X \to \Pic^0(X)$.

\end{NB}

\begin{prop}[{cf. \cite{Or1}}]\label{prop:WIT}
Let $\Phi:{\bf D}(X) \to {\bf D}(Y)$ be an equivalence.
For a semi-homogeneous sheaf $E$ on $X$,
there is an integer $n$ such that $\Phi(E)[n]$ is a semi-homogeneous sheaf. 
\end{prop}

\begin{prop}[{\cite[Prop. 4]{Fourier}}]\label{prop:semi-homog}
Let $E$ and $F$ be semi-homogeneous sheaves.
\begin{enumerate}
\item
Assume that $E$ and $F$ are locally free sheaves.
\begin{enumerate}
\item
If $\langle v(E),v(F) \rangle>0$, then
$\Hom(E,F)=\Ext^2(E,F)=0$.
\item
If $\langle v(E),v(F) \rangle<0$, then
$(\mu(E),H) \ne (\mu(F),H)$, $\Ext^1(E,F)=0$ and
\begin{equation}
\begin{cases}
\Hom(E,F)=0, & (\mu(E),H)>(\mu(F),H)\\
\Ext^2(E,F)=0, & (\mu(F),H)>(\mu(E),H).
\end{cases}
\end{equation}
\end{enumerate}
\item
Assume that $E$ is locally free and $F$ is a torsion sheaf.
\begin{enumerate}
\item
If $\langle v(E),v(F) \rangle>0$, then
$\Hom(E,F)=\Ext^2(E,F)=0$.
\item
If $\langle v(E),v(F) \rangle<0$, then
$\Ext^1(E,F)=\Ext^2(E,F)=0$.
\end{enumerate}
\item
Assume that $E$ and $F$ are torsion sheaves.
Then $\langle v(E),v(F) \rangle \geq 0$.
If $\langle v(E),v(F) \rangle>0$, then 
$\Hom(E,F)=\Ext^2(E,F)=0$.
\end{enumerate}
\end{prop}

\begin{rem}
If $(D^2)>0$, then $D$ is ample if and only if $(D,H_0)>0$ for an ample
divisor $H_0$.
Since $-\langle v(E),v(F) \rangle=\rk E \rk F ( (\mu(E)-\mu(F))^2)/2$,
$\langle v(E),v(F) \rangle<0$ implies $\mu(E)-\mu(F)$ is ample or
$\mu(F)-\mu(E)$ is ample.
\end{rem}

\subsection{Cohomological Fourier-Mukai transforms}\label{subsect:FM}
We collect some results on the Fourier-Mukai
transforms on abelian surfaces $X$ with $\rk \NS(X)=1$.
Let $H_X$ be the ample generator of $\NS(X)$.
We shall describe the action of Fourier-Mukai transforms 
on the cohomology lattices in \cite{YY1}.
For $Y \in \FM(X)$, we have $(H_Y^2)=(H_X^2)$.
We set $D:= (H^2_X)/2$. 
In \cite[sect. 6.4]{YY1}, we constructed an isomorphism of lattices
\begin{equation*}
\begin{matrix}
\iota_X: &
(H^*(X,\bb{Z})_{\alg},\langle\,\,,\,\,\rangle) 
& \simto & 
(\Sym_2(\bb{Z}, D),B),\\
& (r,dH_X,a) &
\mapsto & 
\begin{pmatrix}
 r & d\sqrt{D} \\ d\sqrt{D} & a
\end{pmatrix},
\end{matrix}
\end{equation*}
where $\Sym_2(\bb{Z}, D)$ is given by 
\begin{align*}
 \Sym_2(\bb{Z},D):=
 \left\{\begin{pmatrix} x &y \sqrt{D} \\ y\sqrt{D}&z\end{pmatrix}\, 
 \Bigg|\, x,y,z\in \bb{Z}\right\},
\end{align*}
and the bilinear form $B$ on $\Sym_2(\bb{Z},D)$ is given by 
\begin{align*}
B(X_1,X_2) := 2D y_1 y_2-(x_1 z_2+z_1 x_2)
\end{align*}
for
$X_i =\begin{pmatrix}x_i & y_i \sqrt{D} \\ y_i \sqrt{D} &z_i \end{pmatrix}
 \in\Sym_2(\bb{Z},D)$
($i=1,2$).

Each $\Phi_{X \to Y} \in \Eq_0({\bf D}(X),{\bf D}(Y))$ gives an isometry 
\begin{align}
 \iota_{Y} \circ \Phi^H_{X \to Y} \circ \iota_{X}^{-1}
\in \LieO(\Sym_2(\bb{Z}, D)),
\end{align}
where $\LieO(\Sym_2(\bb{Z}, D))$ is the isometry group
of the lattice $(\Sym_2(\bb{Z}, D),B)$. 
Thus we have a map
$$
\eta:
{\cal E} \to \LieO(\Sym_2(\bb{Z}, D))
$$
which preserves the structures of 
multiplications. 

\begin{defn}\label{defn:G}
We set
\begin{align*}
&\widehat{G} := 
\left\{
 \begin{pmatrix} a \sqrt{r} & b \sqrt{s}\\
c \sqrt{s} & d \sqrt{r}
\end{pmatrix}
 \Bigg|\, 
\begin{aligned}
 a,b,c,d,r,s \in \bb{Z},\, r,s>0\\ 
rs=D, \, adr-bcs = \pm1
\end{aligned}
\right\},
\\
&G := \widehat{G}\cap \SL(2,\bb{R}).
\end{align*}
\end{defn}

We have an action $\cdot$ of $\widehat{G}$ on the lattice
$(\Sym_2(\bb{Z}, D),B)$:
 
\begin{align}
\label{eq:action_cdot}
g \cdot  
\begin{pmatrix}
r &d  \sqrt{D}\\d  \sqrt{D}&a \end{pmatrix}
:= 
  g
 \begin{pmatrix}r &d  \sqrt{D}\\d  \sqrt{D}&a \end{pmatrix} 
{}^t g,\;
g \in \widehat{G}.
\end{align}
Thus we have a homomorphism:
$$
\alpha:\widehat{G}/\{\pm 1\} \to \LieO(\Sym_2({\Bbb Z},D)).
$$

\begin{thm}[{\cite[Thm. 6.16, Prop. 6.19]{YY1}}]\label{thm:FM}
Let
$\Phi \in \Eq_0(\bl{D}(Y),\bl{D}(X))$ be an equivalence.
\begin{enumerate}
\item[(1)]
$v_1:=v(\Phi({\cal O}_Y))$ and $v_2:=\Phi(\varrho_Y)$
are positive isotropic Mukai vectors with 
$\langle v_1,v_2 \rangle=-1$ and we can write 
\begin{equation*}
\begin{split}
& v_1=(p_1^2 r_1,p_1 q_1 H_X, q_1^2 r_2),\quad
v_2=(p_2^2 r_2,p_2 q_2 H_X, q_2^2 r_1),\\
& p_1,q_1,p_2,q_2, r_1, r_2 \in {\Bbb Z},\;\;
p_1,r_1,r_2 >0,\\
& r_1 r_2=D,\;\; p_1 q_2 r_1-p_2 q_1 r_2=1.
\end{split}
\end{equation*}
\item[(2)]
We set
\begin{equation*}
{}^t \theta(\Phi):=\pm
\begin{pmatrix}
p_1 \sqrt{r_1} & p_2 \sqrt{r_2}\\
q_1 \sqrt{r_2} & q_2 \sqrt{r_1}
\end{pmatrix}
\in G/\{\pm 1\}.
\end{equation*}
Then ${}^t \theta(\Phi)$ is uniquely determined by 
$\Phi$ and  
we have a map
\begin{equation*}
{}^t \theta:{\cal E} \to G/\{\pm 1\}.
\end{equation*}
\item[(3)]
The action of ${}^t \theta(\Phi)$ on $\Sym_2({\Bbb Z},D)$
is the action of $\Phi$ on the algebraic Mukai lattice: 
\begin{equation*}
\iota_X \circ \Phi(v)
={}^t \theta(\Phi) \cdot\iota_Y(v).
\end{equation*}
\begin{NB}
\begin{equation*}
\iota_X \Phi\iota_Y^{-1}
\begin{pmatrix}r &d  \sqrt{D}\\d  \sqrt{D}&a \end{pmatrix} 
=
{}^t \theta(\Phi)
\begin{pmatrix}r &d  \sqrt{D}\\d  \sqrt{D}&a \end{pmatrix} 
{}^t({}^t\theta(\Phi))
\end{equation*}
\end{NB}
Thus we have the following commutative diagram:
\begin{align}
\label{diag:groups}
\xymatrix{    
{\cal E}  \ar[d]_-{{}^t \theta} \ar[dr]_{\eta} 
&
\\
\widehat{G}/\{\pm 1\} \ar[r]_-{\alpha}
&
\LieO(\Sym_2({\Bbb Z},D))   
}
\end{align}

\end{enumerate}
\end{thm}
For $\Phi:\Eq_0({\bf D}(X),{\bf D}(X))$ preserving
the sublattice $L:={\Bbb Z} \oplus {\Bbb Z}H_X \oplus {\Bbb Z}\varrho_X$,
it is easy to see that the same proof of Theorem \ref{thm:FM}
works.
Thus replacing $H^*(X,{\Bbb Z})_{\alg}$ by $L$,
we have a similar claims to Theorem \ref{thm:FM}.
From now on, we identify the Mukai lattice
$L$ with $\Sym_2({\Bbb Z},D)$
via $\iota_X$. 
Then for $g \in \widehat{G}$ and $v \in L$,
$g \cdot v $ means
$\iota_X(g \cdot v)=g \cdot \iota_X(v)$.    
We also set $H:=H_X$.
By \cite[Lem. 2.5]{Y}, we have the following.
\begin{lem}
For
\begin{equation}\label{eq:n}
A=\begin{pmatrix}
a & b \sqrt{D}\\
c \sqrt{D} & d
\end{pmatrix},\;a,b,c,d \in {\Bbb Z}, ad-bcD=1,
\end{equation}
there is a Fourier-Mukai transform
$\Phi_{X \to X}^{{\bf E}_A}:{\bf D}(X) \to {\bf D}(X)$ 
such that $\Phi_{X \to X}^{{\bf E}_A}(L)=L$ with
${}^t \theta(\Phi_{X \to X}^{{\bf E}_A})=A$.
\end{lem}
\begin{NB}
For $v_1=(a^2,acH,c^2 D), v_2=(b^2 D,bdH, d^2)$,
we take a family of semi-homogeneous sheaves ${\bf E}_A$
such that $\Phi_{X \to X}^{{\bf E}_A}(v({\cal O}_X))=v_1$ and
$\Phi_{X \to X}^{{\bf E}_A}(\varrho_X)=v_2$.
Then $L$ is preserved.
\end{NB}

Replacing $A$ by $-A$ if necessary,
we assume that $\tr A \geq 0$.
\begin{defn}\label{defn:A}
For a matrix $A$ in \eqref{eq:n} such that $\tr A \geq 0$,  
let $\Phi_{X \to X}^{{\bf E}_A}:{\bf D}(X) \to {\bf D}(X)$ 
such that ${\bf E}_A \in \Coh(X \times X)$ and 
$\Phi_{X \to X}^{{\bf E}_A}(L)=L$ with
${}^t \theta(\Phi_{X \to X}^{{\bf E}_A})=A$.
\end{defn}
We note that 
\begin{equation}\label{eq:v_1}
v(({\bf E}_A)_{|\{x\} \times X})=(b^2 D,bd H, d^2),\;
v(({\bf E}_A^{\vee})_{|X \times \{x \}})=(b^2 D,-ba H, a^2).
\end{equation}

\begin{rem}\label{rem:T}
If ${\bf E}'_A$ also induces the same cohomological transform $\Phi$,
then
there is an isomorphism $f:X \to X$ and a line bundle $P \in \Pic^0(X)$
such that ${\bf E}'_A=(1_X \times f)^*({\bf E}_A) \otimes p_2^*(P)$.

Let $T$ be a subgroup of $\Eq_0({\bf D}(X),{\bf D}(X))$
induced by $\Aut_H(X) \times \Pic^0(X)$, where
$\Aut_H(X)$ is the group of isomorphisms preserving $H$.
Then $T$ preserves $L$ and
$\Phi_{X \to X}^{{\bf E}_A} \mod T$ is determined by $A$.
\end{rem}

\begin{rem}
If $X$ is an abelian surface with $\End(X) \cong {\Bbb Z}$, then
$\NS(X)={\Bbb Z}H$ with $(H^2)=2D$ and 
for all autoequivalences $\Phi$, 
${}^t \theta(\Phi)$ is of the form \eqref{eq:n}.
\end{rem}

\begin{rem}
For an equivalence
$\Phi_{X \to X}^{{\bf E}}:{\bf D}(X) \to {\bf D}(X)$
such that ${\bf E} \in \Coh(X \times X)$ and
$\Phi_{X \to X}^{{\bf E}}$ preserves $L$.
We set 
\begin{equation}
B:={}^t \theta(\Phi_{X \to X}^{{\bf E}})=
\begin{pmatrix}
p_1 \sqrt{r_1} & p_2 \sqrt{r_2}\\
q_1 \sqrt{r_2} & q_2 \sqrt{r_1}
\end{pmatrix}
\end{equation}
as in Theorem \ref{thm:FM} (2).
Then $B^2$ is of the form \eqref{eq:n}.
If $(p_1+q_2)p_2 >0$ or $p_2=0$, then 
$(\Phi_{X \to X}^{{\bf E}_B})^2 \equiv \Phi_{X \to X}^{{\bf E}_{B^2}} \mod T$.
If $(p_1+q_2)p_2<0$ or $p_1+q_2=0$, then
$(\Phi_{X \to X}^{{\bf E}_B})^2 \equiv 
\Phi_{X \to X}^{{\bf E}_{B^2}[-2]} \mod T$.
Hence for the computation of $h_t(\Phi_{X \to X}^{{\bf E}_B})$
is reduced to the computation of $h_t(\Phi_{X \to X}^{{\bf E}_A})$.
In particular, 
if $\NS(X)={\Bbb Z}H$, then we can compute $h_t$ for all equivalences.
\end{rem}

Since the eigen equation of $A$ in \eqref{eq:n} is $x^2-(a+d)x+1=0$,
$A$ has two real eigen value $\alpha>\beta=1/\alpha$
unless $\tr A=0, 1,2$.
If $\tr A=0,1$, then $A^4=E, A^6=E$.
If $\tr A=2$, then $(A- E)^2=0$, and hence
\begin{equation}\label{eq:tr=2}
A^n=E+n(A-E).
\end{equation}  

Assume that $\tr A>2$. Then 
\begin{equation}\label{eq:irrational}
\alpha, \beta \not \in {\Bbb Q},
\end{equation}
since ${\Bbb Z}$ is integrally closed and $0<\beta<1$.
\begin{NB}
\begin{equation}
A^n=\frac{\alpha^n-\beta^n}{\alpha-\beta}A+
\frac{\alpha^n \beta-\alpha \beta^n}{\beta-\alpha}E
\end{equation}
\end{NB}
We set 
\begin{equation}
P:=\frac{1}{\alpha-\beta}(A-\beta E),\;
Q:=\frac{1}{\beta-\alpha}(A-\alpha E).
\end{equation}
Then 
\begin{NB}
\begin{equation}
A=\alpha P+\beta Q,\;P+Q=E,\,PQ=0.
\end{equation}
\end{NB}

\begin{equation}\label{eq:tr>2}
A^n=\alpha^n P+\beta^n Q.
\end{equation}

Let $E$ be a semi-homogeneous sheaf on $X$ with 
$v(E)=(p^2,pqH,q^2 D)$, $p^2,pq, q^2 D \in {\Bbb Z}$. 
Then 
\begin{equation}
\iota_X(v(L))=u {}^t u,\; u=\begin{pmatrix} p\\
q \sqrt{D}
\end{pmatrix}. 
\end{equation}
We set
\begin{equation}
\begin{pmatrix}
p_n\\
q_n \sqrt{D}
\end{pmatrix}=A^n u=\alpha^n Pu+\beta^n Q u.
\end{equation}
Then 
\begin{equation}
\iota_X(v(\Phi^n(E)))=
\begin{pmatrix}
p_n\\
q_n \sqrt{D}
\end{pmatrix}
\begin{pmatrix}
p_n &
q_n \sqrt{D}
\end{pmatrix}.
\end{equation}
Let 
\begin{equation}
u_\alpha=\begin{pmatrix}
1\\
s \sqrt{D}
\end{pmatrix}
\end{equation}
be an eigen vector with respect to $\alpha$.
Thus $s=\frac{\alpha-a}{bD}$.
By \eqref{eq:irrational},
$(A-\beta E) u \ne 0$, and hence
\begin{equation}\label{eq:p-q}
\lim_{n \to \infty} \frac{q_n}{p_n}=s.  
\end{equation}
%
\begin{NB}
If $A$ has eigen value $\alpha,\beta \in {\Bbb C}$,
then the eigen equation for the action on $\Sym_2({\Bbb R},D)$
is 
$(x-\alpha^2)(x-\alpha \beta)(x-\beta^2)=
(x^2-((\tr A)^2-2\det A)x+(\det A)^2)(x-\det A)$. 
\end{NB}
The following is obvious.
\begin{lem}\label{eq:eigen}
Assume that $x^2-\tr A x+\det A=(x-\alpha)(x-\beta)$, $\alpha, 
\beta \in {\Bbb C}$.
Then the eigen equation of the representation matrix of the action of $A$ on 
$\Sym_2({\Bbb R},D)$ is
\begin{equation}
(x-\alpha^2)(x-\alpha \beta)(x-\beta^2)=
(x^2-((\tr A)^2-2\det A)x+(\det A)^2)(x-\det A).
\end{equation} 
In particular, if $|\alpha| \geq |\beta|$, then
$|\alpha|^2$ is the spectral radius of this representation. 
\end{lem}

\begin{NB}
$\Phi$ also preserves $L^\perp$, which is negative definite.
Hence the eigen values are root of unity.

Indeed let $B$ be the intersection matrix.
Then $B={}^t Q Q$.
If ${}^t P B P=B$, then
${}^t (Q PQ^{-1}) (QPQ^{-1})=E$.
Hence $QPQ^{-1}$ is a orthogonal matrix.
Hence the eigen values of $QPQ^{-1}$ and hence of
$P$ are root of unity. 
\end{NB}

\section{Computation of entropy}\label{sect:computation}

For an endofunctor $\Phi:{\bf D}(X) \to {\bf D}(X)$,
let $h_t(\Phi):{\Bbb R} \to \{-\infty\} \cup {\Bbb R}$
be the entropy defined in \cite[Defn. 2.5]{DHKK}.
In this section, we shall compute the entropy  
for the equivalence $\Phi:=\Phi_{X \to X}^{{\bf E}_A}$ 
by using the following result.
\begin{NB}
Let $G$ be a split generator of ${\bf D}(X)$.
 For any $E \in {\bf D}(X)\setminus\{0\}$, 
 we have the following collection of triangles
 \begin{align}\label{diag:HN}
 \xymatrix{
    0=E_0   \ar[rr]         &
  & E_1     \ar[dl] \ar[rr] &
  & E_2     \ar[r]  \ar[dl] & \cdots 
   E_{n-1} \ar[rr]         &
  & E_n =E \oplus E' \ar[dl]
 \\
                            & G[k_1] \ar[ul]^{[1]} 
  &                         & G[k_2] \ar[ul]^{[1]}
  &                        
  &                         & G[k_n] \ar[ul]^{[1]}
 }
 \end{align}
We set
\begin{equation}
\delta_t(G,E):=\inf \left\{\sum_{i=1}^n e^{k_i t} \right\},
\end{equation} 
where $k_i$ are in the diagram \eqref{diag:HN}.
\begin{defn}
For an exact endofunctor $\Phi$ of ${\bf D}(X)$, we define
\begin{equation}
h_t(\Phi):=\lim_{n \to \infty} \frac{1}{n}\log \delta_t(G,\Phi^n(G)).
\end{equation}
\end{defn}

\end{NB}

\begin{prop}[{\cite[Thm. 2.7]{DHKK} and \cite[Prop. 3.8]{KT}}]
Let $G,G'$ be split generators of ${\bf D}(X)$ and $F$ an endofunctor of
${\bf D}(X)$ of Fourier-Mukai type such that $F^n G,F^n G' \not \cong 0$
for all $n>0$. Then 
\begin{equation}
h_t(F)=\lim_{n \to \infty}\frac{1}{n}\log \delta_t'(G,F^n G')
\end{equation} 
 where
\begin{equation}
\delta_t'(M,N):=\sum_{m \in {\Bbb Z}}
\dim \Hom(M,N[m])e^{-mt},\;M,N \in {\bf D}(X).
\end{equation}
\end{prop}

\begin{thm}\label{thm:h_t}
Let $\Phi:=\Phi_{X \to X}^{{\bf E}_A}$ be an equivalence associated to
$A$ (see Definition \ref{defn:A}). 
\begin{enumerate}
\item[(1)]
 Assume $b=0$.  Then $h_t(\Phi)=\log \rho(\Phi)=0$.
\item[(2)]
 Assume that $b>0$. Then
\begin{equation}
h_t(\Phi)=
\begin{cases}
\log \rho(\Phi), & \tr A \geq 2\\
\log \rho(\Phi)-\frac{2}{3}t, & \tr A=1\\
\log \rho(\Phi)-t, & \tr A=0. 
\end{cases}
\end{equation}
\item[(3)]
Assume that $b<0$. Then
\begin{equation}
h_t(\Phi)=
\begin{cases}
\log \rho(\Phi)-2t, & \tr A \geq 2\\
\log \rho(\Phi)-\frac{4}{3}t, & \tr A=1\\
\log \rho(\Phi)-t, & \tr A=0. 
\end{cases}
\end{equation}
\end{enumerate}
\end{thm}

\begin{rem}
If $\tr A \leq 2$, then $\rho(\Phi)=1$.
\end{rem}
 
\begin{proof}
Let $N$ be a line bundle with $c_1(N)=mH$, $m>0$.
We take 
\begin{equation}\label{eq:split-gen}
G=\bigoplus_{-3 \leq i \leq -1} N^{ \otimes i},\;
G'=\bigoplus_{1 \leq i \leq 3} N^{ \otimes i}
\end{equation}
as split generators of ${\bf D}(X)$ (\cite{Or2}).

(1)
Assume that $b=0$.
Then $\Phi_{X \to X}^{{\bf E}_A}$ is an equivalence 
defined by $\Aut(X) \times \Pic(X)$.
Then it is easy to see that
\begin{equation}
h_t(\Phi_{X \to X}^{{\bf E}_A})=0=\log \rho(\Phi_{X \to X}^{{\bf E}_A}).
\end{equation}

\begin{NB}
We note that $\Phi^n(G')$ is a locally free sheaf.
Hence $\delta_t'(G,\Phi^n(G'))=\chi(G,\Phi^n(G'))$.
For an isomorphism $f:X \to X$,
$f^*(H)=H$, and hence $f^*$ acts trivially on the Mukai lattice. 
Hence $\Phi$ acts on the Mukai lattice by 
$e^{dH}$. In particular $\Phi^n(N^{\otimes i})=N^{\otimes i}(dnH)$.
It is easy to see that $\chi(G,\Phi^n(G'))$ is a quadratic polynomial of $n$.
Hence $h_t( \Phi_{X \to X}^{{\bf E}_A})=0$.

\end{NB}

We next prove (2) and (3). 
So we assume that $b \ne 0$.

(I). We first treat the case where $\tr A>2$.
We set 
\begin{equation}
\begin{pmatrix}
p_{i,n}\\
q_{i,n} \sqrt{D}
\end{pmatrix}
=A^n 
\begin{pmatrix}
1 \\
im \sqrt{D}
\end{pmatrix}.
\end{equation}
Then 
$v(\Phi^n(N^{\otimes i}))=(p_{i,n}^2,p_{i,n}q_{i,n}H,q_{i,n}^2 D)$.
For a sufficiently large $n$, 
$q_{i,n}/p_{i,n}$ is sufficiently close to $s$ by \eqref{eq:p-q}.
Hence we can take an integer $m$ such that 
\begin{equation}\label{eq:m}
\frac{q_{i,n}}{p_{i,n}}>-m
\end{equation}
for all sufficiently large $n$ (depending on $m$).
We note that $\mu(({\bf E}_A^{\vee})_{|X \times \{x \}})=-\frac{a}{bD}H$ and
\begin{equation}\label{eq:s}
s+\frac{a}{bD}=\frac{\alpha-a}{bD}+\frac{a}{bD}=\frac{\alpha}{bD}.
\end{equation}
Let $E$ be a semi-homogeneous sheaf with $\mu(E)=xH$.
We set $\mu(\Phi^n(E))=x_n H$.
Then 
\begin{equation}
x_{n+1}=\frac{c+dx_n}{a+bD x_n}.
\end{equation}

Assume that $b>0$.
Then \eqref{eq:s} implies $s>-\frac{a}{bD}$.
If $|s-x|<\frac{\alpha-1}{|b|D}$, then 
\begin{equation}
\left| \frac{c+dx}{a+bD x} -s\right| \leq \frac{|x-s|}{\alpha}.
\end{equation}
Hence $x_n$ also satisfies the same condition.
If $E$ also satisfies $x>-\frac{a}{bD}$, then 
$x_n>-\frac{a}{bD}$ for all $n \geq 0$.
Hence 
$\Phi^n(E) \in \Coh(X)$ for $n \geq 0$ by Proposition \ref{prop:semi-homog}.

We set $\Phi^n(N^{\otimes i})=E_n^i [\phi_i(n)]$, $E_n^i \in \Coh(X)$
(cf. Proposition \ref{prop:WIT}).
For  $E:=E_{n_0}^i$,
$\mu(E)=\frac{p_{i,n_0}}{q_{i,n_0}}H$. Hence 
by \eqref{eq:p-q},
we get $\Phi^{n'}(E_{n_0}^i) \in \Coh(X)$ for $n' \geq 0$ and $n_0 \gg 0$.
Therefore
$\phi_i(n) \in 2{\Bbb Z}$ is constant for a sufficiently large $n$.
Hence there are $l_i$ such that 
$\Hom(G,\Phi^n(N^{\otimes i})[k])=0$ for $k \ne l_i$ and
\begin{equation}
\begin{split}
\delta_t'(G,\Phi^n(G'))=& \sum_{-3 \leq i \leq -1}
\sum_{1 \leq j \leq 3}
\chi(N^{\otimes i},\Phi^n(N^{\otimes j}))e^{-l_j t}\\
=& \sum_{-3 \leq i \leq -1}
\sum_{1 \leq j \leq 3}
(\alpha^n a_{i,j}+\beta^n b_{i,j})^2 e^{-l_j t}
\end{split}
\end{equation}
where
\begin{equation}
\begin{split}
\alpha^n a_{i,j}+\beta^n b_{i,j}
=& 
\det \left(\begin{pmatrix}
1\\
im \sqrt{D}
\end{pmatrix},
 (\alpha^n P+\beta^n Q)\begin{pmatrix}
1\\
jm \sqrt{D} 
\end{pmatrix}
\right)\\
=& \det \left(\begin{pmatrix}
1\\
im \sqrt{D}
\end{pmatrix},
A^n \begin{pmatrix}
1\\
jm \sqrt{D} 
\end{pmatrix}
\right)
\end{split}
\end{equation}
(see \eqref{eq:tr>2}).
By \eqref{eq:m}, $a_{i,j} \ne 0$ for all sufficiently large $n$. 
Hence $\log \delta_t'(G,\Phi^n(G')) \sim 2n \log |\alpha|$ and 

\begin{equation}
h_t(\Phi)=\lim_{n \to \infty} \frac{1}{n} \log \delta_t'(G,\Phi^n(G'))=
2\log |\alpha|=\log \rho(\Phi).
\end{equation}

Assume that $b<0$.
Then \eqref{eq:s} implies $s <-\frac{a}{bD}$.
If $|s-x|<\frac{\alpha-1}{|b|D}$, then 
\begin{equation}
\left| \frac{c+dx}{a+bD x} -s\right| \leq \frac{|x-s|}{\alpha}.
\end{equation}
Hence $x_n$ also satisfies the same condition.
If $E$ also satisfies $x<-\frac{a}{bD}$, then 
$x_n<-\frac{a}{bD}$.
Hence $\Phi(E)[2] \in \Coh(X)$.
We set $(\Phi[2])^n(N^{\otimes i})=E_n^i [\psi_i(n)]$, $E_n^i \in \Coh(X)$.
In this case, $\psi_i(n)$ is constant for a sufficiently large $n$.
Hence there are $l_i$ such that 
$\Hom(G,\Phi^n(N^{\otimes i})[k])=0$ for $k \ne l_i+2n$ and
$\log \chi(G,\Phi^n(G')) \sim 2n \log |\alpha|$.
Hence

\begin{equation}
h_t(\Phi)=\lim_{n \to \infty} \frac{1}{n} \log \delta_t'(G,\Phi^n(G'))=
2\log |\alpha|-2t=\log \rho(\Phi)-2t.
\end{equation}

(II).
Assume that $\tr A=2$.
Then $A^n=E+n(A-E)$.
We set $s:=\frac{1-a}{bD}$.
Let $E$ be a semi-homogeneous sheaf with $\mu(E)=xH$ and 
set $\mu(\Phi^n(E))=x_n H$. 
Since 
\begin{equation}
\begin{split}
x_n=& \frac{nc+(n(d-1)+1)x}{1+n(a-1)+nbD x},\\
x_n -s=& 
\frac{bDx+(a-1)}{(1+n(a-1+bDx))bD},
\end{split}
\end{equation}
we have
$\lim_{n \to \infty} x_n=s$.
In the same way as in the case $\tr A>2$,
we see that

\begin{equation}
h_t(\Phi)=
\begin{cases}
\log \rho(\Phi), & b>0\\
\log \rho(\Phi)-2t, & b<0.
\end{cases}
\end{equation}

(III).
Assume that $\tr A=1$.
Then $A^3=-E$.
It is easy to see that
\begin{equation}
(\Phi^{{\bf E}_A}_{X \to X})^2 \equiv
\begin{cases}
\Phi_{X \to X}^{{\bf E}_{A^2}} \mod T& b>0\\
\Phi_{X \to X}^{{\bf E}_{A^2}}[-2] \mod T & b<0
\end{cases}
\end{equation}
and
\begin{equation}
(\Phi^{{\bf E}_A}_{X \to X})^3 \equiv 
\begin{cases}
[-2] \mod T& b>0\\
[-4] \mod T & b<0
\end{cases}
\end{equation}
(see Remark \ref{rem:T}).
\begin{NB}
We note that 
$\frac{d}{bD}+\frac{a}{bD}=\frac{1}{bD}$.
Hence $\frac{d}{bD}>-\frac{a}{bD}$.
\end{NB}
Indeed  
$\Phi_{X \to X}^{{\bf E}_A}(({\bf E}_A)_{|\{ x\} \times X}) \in M_H(u^{\vee})$
for $b>0$ and
 $\Phi_{X \to X}^{{\bf E}_A}(({\bf E}_A)_{|\{ x\} \times X})[2] 
\in M_H(u^{\vee})$
for $b<0$, where 
$u:=v(({\bf E}_A)_{|X \times \{ x\} \times X})=(b^2 D,ba H,a^2)$.
Hence
\begin{equation}
h_t (\Phi^{{\bf E}_A}_{X \to X})=
\begin{cases}
-\frac{2}{3}t & b>0\\
-\frac{4}{3}t & b<0.
\end{cases}
\end{equation}

Assume that $\tr A=0$. Then 
$(\Phi_{X \to X}^{{\bf E}_A})^2 \equiv [-2] \mod T$ implies
\begin{equation}
h_t (\Phi^{{\bf E}_A}_{X \to X})=-t.
\end{equation}
\end{proof}

\begin{NB}
In particular, \cite[Conj. 5.3]{KT} holds for the equivalence
$\Phi$.
\begin{cor}\label{cor:h}
$h_0(\Phi)=\log \rho(\Phi)$.
\end{cor}
\end{NB}

\begin{NB}
\begin{rem}
Let $X$ be arbitrary abelian surface.
For an autoequivalence $\Phi:{\bf D}(X) \to {\bf D}(X)$
such that $(c_1(\Phi(k_x))^2)>0$, the same claim holds:

We set $v:=v(\Phi(k_x))=(r,dH,a)$, where $H$ is ample. 
$L:={\Bbb Z}\oplus {\Bbb Z}H \oplus {\Bbb Z}\varrho_X$. 
Assume that $\gcd(r,a)=1$, that is,
there is $u \in L$ with $\langle u,v \rangle=1$.
We may assume that $u$ is also isotropic.
\end{rem}
\end{NB}

\begin{rem}
Our assumption of $\Phi(L)=L$ is very strong in general.
Indeed 
$H$ is not preserved by the action of $\Aut(X)$ in general,
and there is an abelian surface with an automorphism of positive entropy.
\end{rem}

\begin{rem}
In \cite{Ikeda}, Ikeda introduced a mass growth 
$h_{\sigma,t}$ of a stability condition $\sigma$
and studied several properties.
In our example, 
\begin{equation}
h_{\sigma, t}(\Phi)=h_t(\Phi),
\end{equation}
where 
$\sigma$ be a stability condition
such that $Z_\sigma(E):=\langle e^{z H},v(E) \rangle$ ($z \in {\Bbb H}$) and
$\phi_\sigma(k_x)=1$:

We only explain the case where $\tr A>2$.
We first note that $\Phi$ preserves $\sigma$-stability 
for semi-homogeneous sheaves.
We set 
\begin{equation}
\alpha^n a_j(z)+\beta^n b_j(z)=\det \left(\begin{pmatrix}
1\\
z \sqrt{D}
\end{pmatrix},
 (\alpha^n P+\beta^n Q)\begin{pmatrix}
1\\
jm \sqrt{D} 
\end{pmatrix}
\right).
\end{equation}
Then by \eqref{eq:tr>2}, we get
$Z_{\sigma}(\Phi^n(N^{\otimes j}))=
(\alpha^n a_j(z)+\beta^n b_j(z))^2$. 
Assume that $b<0$.  Then $\phi_\sigma((\Phi[2])^n(N^{\otimes j}))$ is bounded.
Hence by using \cite[Thm. 1.1]{Ikeda},
we get $h_{\sigma,t}(\Phi)=\log |\alpha|^2-2t=h_t(\Phi)$.
\end{rem}

\subsection{Gromov-Yomdin type conjecture 
by Kikuta and Takahashi}
Let $X$ be an arbitrary abelian surface.
We shall compute
 $h_0(\Phi)$ and check the Kikuta and Takahashi's conjecture
\cite[Conjecture 5.3]{KT}
for some endofunctors of ${\bf D}(X)$.
\begin{lem}\label{lem:D}
Let $D$ be a divisor with $(D^2)>0$ and $\eta \in \NS(X)$.
There is no isotropic vector $v \ne 0$ in 
$({\Bbb Q}e^{\eta}+{\Bbb Q}e^{\eta+D}+{\Bbb Q}e^{\eta+2D})^\perp$.
\end{lem}

\begin{proof}
Replacing $v$ by $ve^{-\eta}$, we may assume that $\eta=0$. 
We set $v=(r,\xi,a) \ne 0$.
Then $(\xi^2)=2ra$.
By the conditions, we get
\begin{equation}
a=0,\; (D,\xi)-r \frac{(D^2)}{2}=0,\;2(D,\xi)-2r(D^2)=0.
\end{equation}
Hence $(D,\xi)=r(D^2)=0$.
Since $(D^2)>0$, we get $r=a=(\xi^2)=0$.
Since $D^\perp$ is negative definite, 
we get $\xi=0$.
Therefore our claim holds.
\end{proof}


\begin{lem}\label{lem:E}
For a semi-homogeneous sheaf $E$ and a line bundle $L$,
\begin{equation}
\sum_{k \in{\Bbb Z}} \dim \Hom(L,E[k])
\leq \max \{4|\chi(L(pH),E)|  \mid p=0,\pm1,\pm2 \},
\end{equation}
where $H$ is an ample divisor on $X$.
\end{lem}

\begin{proof}
We first note that $E$ is a semi-stable sheaf with respect to $H$ by 
Proposition \ref{prop:stable}.
Then $E$ is $S$-equivalent to $\oplus_i E_i$ such that
$E_i$ are stable with $v(E_i)=v(E_j)$.
Hence it is sufficient to prove the claim
for a stable semi-homogeneous sheaf $E$.   
We first assume that $\chi(L,E)=0$.
Then $E$ is not 0-dimensional.
We can easily show the following claim.
\begin{enumerate}
\item
If $(c_1(E \otimes L^{\vee}),H)>0$, then
$\dim \Hom(L,E)=\dim \Hom(L,E[1])$ and
$\Hom(L,E[2])=0$.
\item
If $(c_1(E \otimes L^{\vee}),H)<0$, then
$\dim \Hom(L,E)=0$ and 
$\dim \Hom(L,E[1])=\dim \Hom(L,E[2])$.
\item
If $(c_1(E \otimes L^{\vee}),H)=0$, then
$\dim \Hom(L,E)=\dim \Hom(L,E[2]) \leq 1$, 
$\dim \Hom(L,E[1]) \leq 2$ and $v(L)=v(E)$.
\end{enumerate}
\begin{NB}
(i) If $\dim E=1$, then $\Hom(L,E[2])=0$.

(ii) $E$ must be locally free.

(iii)
If $(c_1(E \otimes L^{\vee}),H)=0$,
then $\chi(E \otimes L^{\vee})=(c_1(E \otimes L^{\vee})^2)/2=0$
implies $c_1(E \otimes L^{\vee})=0$.
Since $E$ is simple, $v(E)=v(L)$.
\end{NB}
For the case of (iii), obviously the claim holds by Lemma \ref{lem:D}.
So we shall treat the case of (i) and (ii).
In these cases, $E$ is not 0-dimensional.
Replacing $H$ by its translate, we can take an injective homomorphism
$E \to E(H)$. Then we get 
$$
\dim \Hom(L,E) \leq \dim \Hom(L(pH),E)
$$
for $p \leq 0$.
We also see that 
$$
\dim \Hom(E,L) \leq \dim \Hom(E,L(pH))
$$
for $p \geq 0$.
By Lemma \ref{lem:D},
 $\chi(L(pH),E) \ne 0$ for an integer $p \in \{0,1,2\}$
and $\chi(L(pH),E) \ne 0$ for an integer $p \in \{0,-1,-2\}$.
Hence by using Proposition \ref{prop:semi-homog}, we get
\begin{equation}
\begin{split}
\dim \Hom(L,E) \leq& \max\{0, \chi(L(pH),E) \mid p=-1,-2\}\\
\dim \Ext^2(L,E) \leq& \max\{0, \chi(L(pH),E) \mid p=1,2\}.
\end{split}
\end{equation}
Thus 
\begin{equation}
\sum_{k \in{\Bbb Z}} \dim \Hom(L,E[k])
\leq \max \left\{2\left|\chi(L(pH),E)\right| 
\mid p=0,\pm1,\pm2 \right\}.
\end{equation}
By Proposition \ref{prop:semi-homog},
the same claim also holds if $\chi(L,E) \ne 0$.
\end{proof}

\begin{prop}\label{prop:endofunctor}
Let $\Phi:{\bf D}(X) \to {\bf D}(X)$ be an endofunctor which is a composite
of equivalences, $f^*$ and $f_*$, where $f:X \to X$
is a finite morphism. Then  
$h_0(\Phi) \leq \log \rho(\Phi)$.
\end{prop}

\begin{proof}
We use split generators $G,G'$ in \eqref{eq:split-gen}.
For a finite morphism $f:X \to X$,
$f^*$ and $f_*$ send semi-homogeneous sheaves to semi-homogeneous sheaves.
\begin{NB}
Indeed for a semi-homogeneous vector bundle, the claim is in 
\cite[Prop. 5.4]{Mu0}. For an elliptic curve $C$ in $X$,
$f^{-1}(C)$ and $f(C)$ are disjoint union of elliptic curves.
Since a semi-homogeneous sheaf of dimension 1 is a semi-homogeneous
vector bundle on an elliptic curve, the claim also holds by 
\cite[Prop. 5.4]{Mu0}.
\end{NB}
By Proposition \ref{prop:WIT},
autoequivaences also send semi-homogeneous sheaves to semi-homogeneous sheaves
up to shift.
Hence $\Phi^n(N^{\otimes j})$ is a semi-homogeneous sheaf
up to shift.
\begin{NB}
\cite[Thm. 1.4]{Fourier}.
\end{NB}
By Lemma \ref{lem:E}, we have
\begin{equation}\label{eq:N-estimate}
\sum_{k \in{\Bbb Z}} \dim \Hom(N^{\otimes i},\Phi^n(N^{\otimes j})[k])
\leq \max \{4|\chi((N(pH))^{\otimes i},\Phi^n(N^{\otimes j}))|
\mid p=0,\pm1,\pm2 \}.
\end{equation}
For any real number $\lambda>\rho(\Phi)$, we have 
\begin{equation}
\lim_{n \to \infty} \frac{1}{\lambda^n}\Phi^n(N^{\otimes j})=0,
\end{equation}
and hence 
\begin{equation}
\lim_{n \to \infty} 
\frac{1}{\lambda^n}|\chi((N(pH))^{\otimes i},\Phi^n(N^{\otimes j})|=0.
\end{equation}
Then we have $|\chi((N(pH))^{\otimes i},\Phi^n(N^{\otimes j})| \leq \lambda^n$ 
for sufficiently large $n$. 
Combining \eqref{eq:N-estimate} with this estimate, we see that 
\begin{equation}
h_0(\Phi)=\lim_{n \to \infty} \frac{1}{n} \log \delta_0'(G,\Phi^n(G'))
\leq \log \lambda.
\end{equation}
Since $\lambda$ is arbitrary, we get
$h_0(\Phi) \leq \log \rho(\Phi)$.
\end{proof}

\begin{rem}
Let $X$ be a simple abelian variety of $\dim X=d$, 
that is, there is no subabelian variety $Y$ 
of $X$ with $0<\dim Y<d$.
Then for any line bundle $L$, $K(L):=\{x \in X \mid T_x^*(L) \cong L \}$
is a finite set, unless $L \in \Pic^0(X)$.
Hence if $c_1(L) \ne 0$, then 
$(L^d) \ne 0$.
Then for a semi-homogeneous sheaf $E$, we see that
$\chi(E) \ne 0$ or $\ch(E) \in {\Bbb Z}_{>0}\ch({\cal O}_X)$.
\begin{NB}
\begin{proof}
Since simple semi-homogeneous vector bundle
is a direct imege of a line bundle by an isogeny,
for any semi-homogeneus vector bundle $E$,
$\chi(E) \ne 0$ or $\ch(E) \in {\Bbb Z} \ch({\cal O}_X)$.
For  a semi-homogeneous sheaf $E$,
applying a Fourier-Mukai transform $\Phi_{X \to \Pic^0(X)}^{{\bf P}}$
to $E(nH)$,
we get a semi-homogeneous vector bundle.
Since $\chi(\Phi_{X \to \Pic^0(X)}^{{\bf P}}(E(nH)))=\rk E$,
$E$ is locally free or 0-dimensional.
In particular, for a semi-homogeneous sheaf $E$,
$\chi(E) \ne 0$ or $\ch(E) \in {\Bbb Z}_{>0}\ch({\cal O}_X)$.
\end{proof}
\end{NB}
Hence we also get
$h_0(\Phi) \leq \log \rho(\Phi)$.  
\end{rem}

\begin{NB}
If $f:X \to X$ is not finite, then
${\bf L} f^* (k_x)$ is not a sheaf if $x \in f(X)$ and
${\bf L} f^* (k_x)=0$ if $x \not \in f(X)$.
We also see that ${\bf R}f*({\cal O}_X)$ is not a sheaf.
Indeed 
$f_*({\cal O}_X)={\cal O}_{f(X)}$ and 
$H^2({\cal O}_X)=H^1(R^1f_*({\cal O}_X))$.
\end{NB}

By \cite[Thm. 1.2]{Ikeda}, $h_0(\Phi) \geq \log \rho(\Phi)$.
Therefore we get the following result which support
a Gromov-Yomdin type conjecture in
\cite[Conjecture 5.3]{KT}.
\begin{prop}\label{prop:KT}
Let $X$ be an abelian surface and
$\Phi:{\bf D}(X) \to {\bf D}(X)$ an endofunctor in 
Proposition \ref{prop:endofunctor}. Then  
$h_0(\Phi) = \log \rho(\Phi)$.
\end{prop}

\section{An example of automorphism on the moduli of stable sheaves}

Let $M_H(v)$ be the moduli space of stable sheaves $E$ on $X$ with
$v(E)=v$ and $K_H(v)$ be a fiber of the albanese map
$a:M_H(v) \to \Alb(M_H(v))=X \times \Pic^0(X)$.
$K_H(v)$ is an irreducible symplectic manifold which is 
derormation equivalent to a generalized Kummer manifold \cite{Y:7}. 
In \cite[Prop. 3.50]{Y}, we constructed an example of moduli space
$M_H(v)$ which have an automorphism of infinite order.
Thus there  is a Fourier-Mukai transform $\Phi$ which induces
an isomorphism $g:M_H(v) \to M_H(v)$ such that $g$ is infinite order.
In this example,
it is easy to see that 
a similar claim to \cite{Ouchi} holds. Thus by using \cite{O}, we get
\begin{equation}
\frac{\dim K_H(v)}{2}h_0(\Phi)=h(g')
\end{equation} 
where $g':K_H(v) \to K_H(v)$ is an automorphism induced $g$
by a Fourier-Mukai
transform $\Phi:{\bf D}(X) \to {\bf D}(X)$ in \cite[Prop. 3.50]{Y}:

Let ${\cal E}$ be a quasi-universal family on $M_H(v) \times X$. 
By $T_x^*({\cal E}) \otimes P_y$ $(x \in X, P_y \in \Pic^0(X))$,
we have an isomorphism $\psi:M_H(v) \to M_H(v)$ such that
$(\psi \times 1_X)^*({\cal E}) \otimes L \cong T_x^*({\cal E}) \otimes P_y
\otimes L' $, where
$L,L'$ are pull-backs of locally free sheaves on $M_H(v)$.
Then 
\begin{equation}
\begin{split}
c_1(p_{M_H(v)!}(\ch((\psi \times 1_X)^*({\cal E})) \alpha^{\vee}))=&
c_1(p_{M_H(v)!}(\ch(T_x^*({\cal E}) \otimes P_y)  \alpha^{\vee}))\\
=&
c_1(p_{M_H(v)!}(\ch({\cal E})  T_{-x}^*(\alpha^{\vee})))\\
=&
c_1(p_{M_H(v)!}(\ch({\cal E}) \alpha^{\vee}))
\end{split}
\end{equation}
for $\alpha \in v^\perp$.
Thus $\psi^* (\theta_v(\alpha))=\theta_v(\alpha) $.
Hence for an isomorphism $g:M_H(v) \to M_H(v)$ induced by $\Phi$,
we have an isomorphism $g':K_H(v) \to K_H(v)$ 
which induces the isomorphism $\Phi:v^\perp \to v^\perp$. 
\begin{NB}
We take $(x,y) \in X \times \Pic^0(X)$ such that
$\psi \circ g$ preserves a fiber of $a$. 
\end{NB}

\begin{NB}
We have a decomposition $V=\sum_i V_i$ such that
$(f-\alpha_i)^{n_i} V_i=0$ $(\alpha_i \in {\Bbb R})$
or $(f-\alpha_i)^{n_i} (f-\overline{\alpha_i})^{n_i} V_i=0$ 
$(\alpha_i \in {\Bbb C} \setminus {\Bbb R})$.
In particular the eigenvalues are $\alpha_i,\overline{\alpha_i}$.
We assume that $|\alpha_1| \geq |\alpha_2| \geq \cdots \geq |\alpha_k |$.
We take a non-zero vector $w=(w_0,w_1,w_2) \in (\oplus_{i>1} V_i)^\perp$.
For $\xi \in \Amp(X)$ such that
$(w_1,\xi) \ne 0$, 
$e^{n \xi} \not \in w^\perp$ except finitely many $n \in {\Bbb Z}$.
\begin{NB2}
If $e^{n_i \xi} \in w^\perp$ for $i=1,2,3$, then
by using Vandermonde determinant, we get 
$w_0 (\xi^2)=(w_1,\xi)=w_2=0$.
\end{NB2}

Since $\rk H^*(X,{\Bbb Z})_{\alg} \leq 6$ and $\det f=\pm 1$,
if $\alpha_1$ is not real, then $|\alpha_2| \leq 1$.
There are 

Assume that $f(v)=\alpha_i v$. Then 
$\langle f(v)^2 \rangle =\langle v^2 \rangle$ implies $\alpha_i=\pm 1$ or
$\langle v^2 \rangle=0$.

\begin{equation}
\begin{split}
\delta_0'(G,\Phi^n(G'))=& \sum_{-3 \leq i \leq -1}
\sum_{1 \leq j \leq 3}
|\chi(N^{\otimes i},\Phi^n(N^{\otimes j}))|\\
=& \sum_{-3 \leq i \leq -1}
\sum_{1 \leq j \leq 3}
(\alpha^n a_{i,j}+\beta^n b_{i,j})^2 
\end{split}
\end{equation}
Indeed we set $v=(r,\xi,a)$. If $r=0$, then $(\xi^2)=0$.

Assume that $X$ is simple, that is, there is no elliptic curve.
For a semi-homogeneous sheaf $E$, 
$\chi(N,E)=0$ implies $(c_1(E \otimes N^{\vee})^2)=0$.
Hence  $v(E) \in {\Bbb Z}v(N)$.
If $E$ is stable, then $\sum_k \dim \Hom(N,E[k]) \leq 4$.

If $\chi(N^{\otimes i},\Phi^n(N^{\otimes j}))=0$, 
then $\sum_k \dim \Hom(N^{\otimes i},\Phi^n(N^{\otimes j}))$

 \end{NB}

\section{Appendix}

Let $(X,H)$ be a principally polarized 
abelian variety of $\dim X=d$.
In \cite{MP}, cohomological action of the group $G$ of Fourier-Mukai transforms
generated by $\Phi_{X \to X}^{{\bf P}}$ and $\otimes {\cal O}_X(H)$
is described. In particular the action on the cohomology group
generated by $H$ is the action of $\SL(2,{\Bbb Z})$ 
on the $d$-th symmetric product of 
${\Bbb Q}^2$.
Then we can also compute $h_t(\Phi_{X \to X}^{{\bf E}})$ of 
$\Phi_{X \to X}^{{\bf E}} \in G$,
where ${\bf E}$ is a coherent sheaf on $X \times X$. 
In particular 
$h_t(\Phi_{X \to X}^{{\bf E}})=\log |\alpha|^d-dt$ if 
$\tr A<-2$.
\begin{rem}
For a semi-homogeneous sheaf $E$ with $\rk E>0$,
if $c_1(E)$ is ample, then $H^i(E)=0$ for $i \ne 0$
 (cf. \cite[Prop. 7.3]{Mu0}).
\end{rem}

\vspace{1pc}

{\it Acknowledgement.}
I would like to thank Atsushi Takahashi for 
useful comments.

\end{document}